\newcount\timehh\newcount\timemm
\timehh=\time
\divide\timehh by 60 \timemm=\time
\documentclass[12pt,letterpaper,reqno]{amsart}

\usepackage{times}
\usepackage[T1]{fontenc}
\usepackage{mathrsfs}

\usepackage{latexsym}
\usepackage[dvips]{graphics}
\usepackage{epsfig}
\usepackage{amsmath,amsfonts,amsthm,amssymb,amscd}
\input amssym.def
\input amssym.tex

\newsavebox\EllipticCurve
\savebox{\EllipticCurve}(200,180)[cc]{%
\setlength{\unitlength}{.6pt}
\begin{picture}(200,180)
\thicklines
\put(0,90){\begin{picture}(200,90)
   \qbezier(18,0)(18,22.5)(54,22.5)
   \qbezier(54,22.5)(69.75,20.25)(85.5,18)
   \qbezier(85.5,18)(101.25,13.5)(117,13.5)
   \qbezier(117,13.5)(153,13.5) (189,85.5)
\end{picture}
} \put(0,90){\begin{picture}(200,90)
   \qbezier(18,0)(18,-22.5)(54,-22.5)
   \qbezier(54,-22.5)(69.75,-20.25)(85.5,-18)
   \qbezier(85.5,-18)(101.25,-13.5)(117,-13.5)
   \qbezier(117,-13.5)(153,-13.5) (189,-85.5)
\end{picture}
}
\end{picture}
}

\newcommand{\hphi}{\widehat{\phi}}

\newcommand\be{\begin{equation}}
\newcommand\ee{\end{equation}}
\newcommand\bea{\begin{eqnarray}}
\newcommand\eea{\end{eqnarray}}
\newcommand\bi{\begin{itemize}}
\newcommand\ei{\end{itemize}}
\newcommand\ben{\begin{enumerate}}
\newcommand\een{\end{enumerate}}

\newtheorem{thm}{Theorem}[section]

\newtheorem{defi}[thm]{Definition}

\newtheorem{rek}[thm]{Remark}

\newcommand{\twocase}[5]{#1 \begin{cases} #2 & \text{#3}\\ #4
&\text{#5} \end{cases}   }

\newcommand{\R}{\ensuremath{\mathbb{R}}}
\newcommand{\C}{\ensuremath{\mathbb{C}}}
\newcommand{\Z}{\ensuremath{\mathbb{Z}}}
\newcommand{\Q}{\mathbb{Q}}

\newcommand{\foh}{\frac{1}{2}}  

\newcommand{\mattwo}[4]
{\left(\begin{array}{cc}
                        #1  & #2   \\
                        #3 &  #4
                          \end{array}\right) }

\newcommand{\gep}{\epsilon}

\newcommand{\notdiv}{\nmid}               

\numberwithin{equation}{section}

\begin{document}

\title{Towards an `average' version of the Birch and Swinnerton-Dyer Conjecture}


\author{John Goes}\email{johnwgoes@gmail.com}
\address{Department of Mathematics, Statistics and Computer Science, University of Illinois at Chicago, Chicago, IL 60680}

\author{Steven J. Miller}\email{Steven.J.Miller@williams.edu}
\address{Department of Mathematics and Statistics, Williams College, Williamstown, MA 01267}


\date{\today}

\thanks{This work was done at the 2009 SMALL Undergraduate Research Project at Williams College, funded by NSF Grant DMS0850577 and Williams College; it is a pleasure to thank them for their support. We would also like to thank Michael Greenblatt, Mark Watkins, and the organizers and participants of the 2009 YMC at Ohio State for helpful discussions on an earlier draft. The second named author was also partly supported by NSF Grant DMS0855257.}

\begin{abstract}
The Birch and Swinnerton-Dyer conjecture states that the rank of the Mordell-Weil group of an elliptic curve $E$ equals the order of vanishing at the central point of the associated L-function $L(s,E)$. Previous investigations have focused on bounding how far we must go above the central point to be assured of finding a zero, bounding the rank of a fixed curve or on bounding the average rank in a family. Mestre \cite{Mes} showed the first zero occurs by $O(1 / \log \log N_E)$, where $N_E$ is the conductor of $E$, though we expect the correct scale to study the zeros near the central point is the significantly smaller $1 / \log N_E$. We significantly improve on Mestre's result by averaging over a one-parameter family of elliptic curves, obtaining non-trivial upper and lower bounds for the average number of normalized zeros in intervals on the order of $1 / \log N_E$ (which is the expected scale). Our results may be interpreted as providing further evidence in support of the Birch and Swinnerton-Dyer conjecture, as well as the Katz-Sarnak density conjecture from random matrix theory (as the number of zeros predicted by random matrix theory lies between our upper and lower bounds). These methods may be applied to additional families of $L$-functions.
\end{abstract}

\maketitle


\section{Introduction}

The goal of this paper is to provide evidence towards the Birch and Swinnerton-Dyer conjecture in one-parameter families of elliptic curves. We briefly summarize our results, assuming the reader is familiar with the notation and subject. Afterwards we review the needed background material from elliptic curves and previous results in \S\ref{sec:ellcurvereview}; for the convenience of the reader, we state all the conjectures assumed or discussed at various points in Appendix \ref{sec:standardconj}. We then prove our theorems and discuss generalizations to other families of $L$-functions in \S\ref{sec:lowerboundnumzeros}, where we give explicit non-trivial upper and lower bounds.

The Birch and Swinnerton Dyer conjecture asserts that if $E$ is an elliptic curve whose Mordell-Weil group $E(\Q)$ has geometric rank $r$, then the associated completed $L$-function $\Lambda(s,E)$ has analytic rank $r$ (i.e., it vanishes to order $r$ at the central point). This is an exceptionally hard problem to investigate, theoretically and numerically. While there is some theoretical evidence when the rank is at most 1, the general case is intractable both theoretically and experimentally. For example, although we can construct elliptic curves with geometric rank exceeding 20, the largest known lower bound for the analytic rank of a $\Lambda(s,E)$ is only 3.\footnote{The number of terms needed for the computation is on the order of the square-root of the conductor of $E$, which grows rapidly in families. While it is possible to numerically show that the first $r$ Taylor coefficients of $\Lambda(s,E)$ are close to zero for many $E$'s with geometric rank $r$, in general these computations can only provide evidence. The exception is when we have formulas for the derivatives as a known quantity times a rational, in which case we can convert these calculations to proofs of vanishing. See \texttt{http://web.math.hr/$\sim$duje/tors/rk28.html} for an example by N. Elkies of an elliptic curve with geometric rank at least 28.}

We consider the following natural question. Let $E$ be an elliptic curve with geometric rank $r$, and assume the Generalized Riemann Hypothesis (GRH). The Birch and Swinnerton-Dyer conjecture predicts that there should be $r$ zeros at the central point. \emph{How far must we go along the critical line before we are assured of seeing $r$ zeros?}

If $N_E$ denotes the conductor of the elliptic curve, we expect the correct scale for zeros near the central point to be of size $1/\log N_E$. Miller \cite{Mil3} investigated the first few zeros above the central point for the family of all elliptic curves as well as one-parameter families of small rank over $\Q(T)$. His results are consistent with the low zeros being of height on the order of $1/\log N_E$; however, the first few zeros are higher than the $N_E \to \infty$ scaling limits predicted by the independent model of random matrix theory. The data suggests that, for finite conductors, better agreement is obtained by modeling these zeros with the interaction model (which involves Jacobi ensembles). Determining the correct corresponding random matrix ensemble involves understanding the discretization of the central values of $L$-functions and the lower order terms in the 1-level density. In his thesis Duc Khiem Huyn \cite{Huy} successfully modeled the first zero of the family of quadratic twists of a fixed elliptic curve, and current work by the second named author and Eduardo Due\~nez, Duc Khiem Huynh, Jon Keating and Nina Snaith is investigating the case of a general one-parameter family \cite{DHKMS}.

The best theoretical result on the first zero above the central point is due to Mestre. Assuming the Generalized Riemann Hypothesis, Mestre \cite{Mes} bounded the analytic rank of $E$ by $O(\log N_E/\log\log N_E)$ and showed its first zero above the central point is at most $B/\log\log N_E$. While this is significantly larger than what we expect the truth to be, namely $O(1/\log N_E)$, it has the advantage of holding for all elliptic curves.

In this note we show that we may reduce the window on the critical line to something of the expected order if we average over a one-parameter family of elliptic curves. Specifically, consider a one-parameter family $\mathcal{E}: y^2 = x^3 + A(T)x + B(T)$ of geometric rank $r$ over $\Q(T)$, with $A(T), B(T) \in \Z[T]$. For each $t\in\Z$ we may specialize and obtain an elliptic curve $E_t: y^2 = x^3 + A(t)x + B(t)$ with conductor $N_t := N_{E_t}$. By Silverman's specialization theorem \cite{Sil2}, for all $t$ sufficiently large each elliptic curve $E_t$ has geometric rank at least $r$.  Assuming standard conjectures, Helfgott \cite{He} proved that for a generic family the sign of the functional equation is 1 half the time and -1 the other half. It is believed that a generic curve in a generic family has analytic rank as small as possible consistent with all constraints. In our case, as the rank must be at least $r$ if the Birch and Swinnerton-Dyer conjecture is true, we expect that in the limit half the curves will have analytic rank $r$ and the other half $r+1$, for an average rank of $r+\frac12$.

We take our family to be $\mathcal{F}_R := \{\Lambda(s,E_t): R \le t \le 2R\}$ with $R\to\infty$, though we often abuse notation and use $\mathcal{F}_R$ to denote $t$ in $[R, 2R]$.  There are two ways to normalize the zeros of $\Lambda(s,E_t)$ near the central point: (1) globally, using $\frac{\log N}{2\pi} :=  \frac1{R}\sum_{t\in\mathcal{F}_R} \frac{\log N_t}{2\pi}$; (2) locally, using $\frac{\log N_t}{2\pi}$. It is significantly easier to use the global rescaling; however, as each elliptic curve can be considered independent of the family, it is more correct to use the local rescaling (\emph{in this case, due to the technicalities that arise we must add some additional restrictions on which $t \in [R,2R]$ are in the family}).

Before stating our main result, we must first introduce some notation. All conjectures are stated in full in Appendix \ref{sec:standardconj}.

\begin{defi}[Sieved family]\label{defi:sievedfamily} Let $\mathcal{E}: y^2 = x^3 + A(T)x+B(T)$ be a one-parameter family of elliptic curves over $\Q(T)$ with discriminant $\Delta(T)$, let $D(T)$ be the product of the
irreducible polynomial factors of the discriminant, and let $B$ be the largest square
dividing $D(t)$ for all integers $t$. For a fixed $c, t_0$, our family is the set of all $t = ct' + t_0$ (with $t \in [R, 2R]$) such that $D(ct'+t_0)$ is square-free except for primes $p|B$ where the power of such $p|D(t)$ is independent of $t$. In \cite{Mil2} it is shown that for any one-parameter family, there is a choice of $c$ and $t_0$ such that the number of such $t$ is $c_{\mathcal{E}}R + o(R)$ for some $c_{\mathcal{E}} > 0$ if every irreducible polynomial factor of $\Delta(T)$ has degree at most 3 (if not, the claim is true if we assume either the ABC or Square-free Sieve Conjecture). We let $\mathcal{F}_R'$ denote the sieved family.
\end{defi}

\begin{defi}[Average number of zeros in a family]\label{defn:Zavelocglob} Let $\mathcal{E}: y^2 = x^3 + A(T)x+B(T)$ be a one-parameter family of elliptic curves over $\Q(T)$ with specialized curves $E_t$ with conductors $N_t$. Assume GRH and write the non-trivial zeros of $\Lambda(s,E_t)$ as $\frac12+i\gamma_{t,j}$, and set \be \frac{\log N}{2\pi} \ := \ \frac1{R} \sum_{t=R}^{2R} \frac{\log N_t}{2\pi}. \ee
The average number of zeros with imaginary part at most $\tau$ (in absolute value) under the global and local renormalizations are defined to be \bea Z_{{\rm ave},\mathcal{E},R}^{({\rm global})}(\tau) & \ := \ & \frac1{R} \sum_{t=R}^{2R} \#\left\{j: \gamma_{t,j} \frac{\log N}{2\pi} \in [-\tau, \tau]\right\} \nonumber\\ Z_{{\rm ave},\mathcal{E},R}^{({\rm local})}(\tau) & \ := \ & \frac1{|\mathcal{F}_R\ '|} \sum_{t=R \atop t \in \mathcal{F}_R'}^{2R} \#\left\{j: \gamma_{t,j} \frac{\log N_t}{2\pi} \in [-\tau, \tau]\right\}. \eea

\end{defi}

The Birch and Swinnerton-Dyer conjecture implies that, for families where half the curves have even and half have odd sign,
\be Z_{{\rm ave},\mathcal{E},R}^{({\rm global})}(0)\ =\ Z_{{\rm ave},\mathcal{E},R}^{({\rm local})}(0)\ \ge\ r + \frac12. \nonumber \ee


Our main results are upper and lower bounds for how many normalized zeros there are on average in the interval $[-\tau, \tau]$, in particular, how small we may take $\tau$ and be assured on average that there are $r + \frac12$ zeros in the interval.

%
%

\begin{thm}\label{thm:mainbound} Let $\mathcal{E}$ be a one-parameter family of elliptic curves of geometric rank $r$ over $\Q(T)$; if $\mathcal{E}$ is not a rational surface (see Remark \ref{rek:commentstate} for a definition) then assume Tate's conjecture. Additionally, if we are using the local renormalization of the zeros we must assume either the ABC or the Square-free Sieve Conjecture if the discriminant has an irreducible polynomial factor of degree at least 4.

Let $\sigma$ be chosen such that we can compute the $1$-level density (defined in \S\ref{sec:oneleveldensity}) for even Schwartz test functions $\phi$ with ${\rm supp}(\hphi) \subset (-\sigma, \sigma)$; see Theorem \ref{thm:millerthesis} for details on what $\sigma$ are permissible for a given family.

Then

\bi

\item \textbf{Lower bounds for the average number of normalized zeros in $[-\tau, \tau]$:} Let the notation be as in Definition \ref{defn:Zavelocglob}, and assume GRH. Let $h$ be any even, twice continuously differentiable function supported on $[-1, 1]$ and monotonically decreasing on $[0,1]$. For fixed $\tau > 0$ let $f(y) = h(2y/\sigma)$, $g(y) = (f \ast f)(y)$ (the convolution of $f$ with itself), and let $\phi(x)$ equal the Fourier transform of $g(y) + (2\pi\tau)^{-2} g''(y)$. Note ${\rm supp}(\widehat{\phi}) \subset (-\sigma, \sigma)$ and $\phi(x)$ is non-negative for $|x| < \tau$ and non-positive for $|x| > \tau$. Then  \be Z_{{\rm avg},\mathcal{E},R}^{({\rm global})}(\tau),\ Z_{{\rm avg},\mathcal{E},R}^{({\rm local})}(\tau) \ \ \ge \ \left(r + \frac12\right) + \frac{\widehat{\phi}(0)}{\phi(0)} + O\left(\frac{\log\log R}{\phi(0)\log R}\right), \ee where \be \frac{\widehat{\phi}(0)}{\phi(0)} \ = \
\frac{(\int_0^1 h(u)^2 du)+(\frac{1}{\sigma\tau\pi})^2(\int_0^1 h(u)h''(u) du)}{\sigma(\int_0^1 h(u)du)^2}.\ee
If we let $\tau_{\rm BSD}(\sigma)$ denote the value of $\tau$ such that we are assured of at least $r+\frac12$ zeros on average (as $R\to\infty$) in $[-\tau, \tau]$ given that we can compute the 1-level density for test functions whose Fourier transform is supported in $(-\sigma, \sigma)$, then \be \tau_{\rm BSD}(\sigma) \ \le \ \frac{1}{\pi}\left(-\frac{\int_0^1 h(u)^2du}{\int_0^1 h(u)h''(u)du}\right)^{-1/2}\ \frac1{\sigma} \ := \ \frac1{\pi C(h)\sigma}. \ee This should be compared to the predictions from the Birch and Swinnerton-Dyer and Parity Conjectures for a generic family, which predict $\tau_{{\rm BSD}}(\sigma) = 0$. In particular, taking \be \twocase{h(x)\ =\ }{(1-x^2)(1-0.233428x^2 + 0.0189588x^4)}{if $|x| \le 1$}{0}{otherwise}\ee  yields \be \tau_{{\rm BSD}}(\sigma)\ \le\ \frac{1}{\pi C(h)\sigma}, \ee where $C(h) \approx 0.63662$ (which is approximately $2/\pi$); note $1/\pi C(h)\sigma$ is approximately $1/2\sigma$. In the arguments below we use $2/\pi$ for brevity without reminding the reader that the numerical calculation is only close to the above. \\

\item \textbf{Upper bounds for the average number of normalized zeros in $[-\tau, \tau]$:}
Let $\psi$ be a twice continuously differentiable even Schwartz test function with ${\rm supp}(\widehat{\psi}) \subset (-\sigma, \sigma)$, $\psi(x) \geq 0$ for all $x$, and $\psi(x)$ monotonically decreasing on $\left[0,\tau\right)$. Then
\bea & &  Z_{{\rm ave},\mathcal{F},R}^{({\rm global})}(\tau),\ Z_{{\rm ave},\mathcal{F},R}^{({\rm local})}(\tau)\nonumber\\ & & \ \ \ \ \ \ \ \ \ \le \
\left(r+\foh\right)+\frac{(r+\foh)(\psi(0)-\psi(\tau)) + \hat{\psi}(0)}{\psi(\tau)} + O\left(\frac{\log\log R}{\psi(0)\log R}\right). \eea If we consider the interval $(-\frac{1}{2\sigma}, \frac{1}{2\sigma})$ from the lower bound,  taking $\psi(x)= \left( \frac{\sin{x\pi \sigma}}{x\pi\sigma}\right)^{2}$ yields the average number of normalized zeros in the limit in this interval is at most $\Big(r + \frac12$ $+$ $\frac1{\sigma}\Big) / \psi(1/2\sigma) = \frac{\pi^2}{4}$ $\Big(r + \frac12$ $+$ $\frac1{\sigma}\Big)$.


\item \textbf{Random matrix theory prediction.} Let $\mathcal{E}$ be a generic one-parameter family of elliptic curves of rank $r$ over $\Q(T)$ with half of the specialized functional equations even and half odd. Assuming the Katz-Sarnak Density Conjecture, as $R\to\infty$ the average number of normalized zeros in $[-\tau, \tau]$ is $(r + \frac12) + 2\tau$; more precisely, random matrix theory predicts \be \lim_{R\to\infty} Z_{{\rm ave},\mathcal{F},R}^{({\rm global})}(\tau),\ Z_{{\rm ave},\mathcal{F},R}^{({\rm local})}(\tau) \ = \ r + \frac12 + 2\tau. \ee In particular, setting $\tau = \frac{1}{2\sigma}$ yields a prediction of $r + \frac12 + 2 \cdot \frac{1}{2\sigma}$ normalized zeros in the limit on average.

\ei

\ \\

In summary, the number of normalized zeros on average as $R\to\infty$ in the interval $\left(-\frac{1}{2\sigma}, \frac{1}{2\sigma}\right)$ satisfy \be r + \frac12 \ \le \ Z_{{\rm ave},\mathcal{E},R}^{({\rm global})}(\tau), \ Z_{{\rm ave},\mathcal{E},R}^{({\rm local})}(\tau) \ \le \ \frac{\pi^2}{4}\left( r + \frac12 + \frac1{\sigma}\right),\ee and this interval contains the prediction from Random Matrix Theory, $r+\frac12+\frac1{\sigma}$.
\end{thm} 


\begin{rek} We obtained our upper bound for $\tau_{\rm BSD}(\sigma)$ by setting $\widehat{\phi}(0)/\phi(0)=0$. The important item to note is that $\tau_{\rm BSD}(\sigma)$ (or any $\tau$) is inversely proportional to the support  $\sigma$. In other words, the larger we may take $\sigma$, the more we may concentrate $\phi$ near the central point and thus the smaller the window. Random matrix theory predicts we may take $\sigma$ arbitrarily large, which would imply we may take $\tau$ arbitrarily small and thus prove the Birch and Swinnerton-Dyer conjecture on average.
\end{rek}


\section{Background material and previous results}\label{sec:ellcurvereview}

\subsection{Elliptic curves}

We quickly review the needed background material on elliptic curves; the reader familiar with the notation and theory may safely skip this subsection. See \cite{Kn,Kob,Sil1,ST} for proofs, as well as the survey \cite{Yo1}.

Let $E$ be an elliptic curve over $\Q$, say $y^2 = x^3 + ax + b$ with $a,b \in \Z$, and set \be E(\Q) \ := \ \{(x,y) \in \Q^2: y^2 = x^3 + ax + b\}. \ee We can define addition of two elements of $E(\Q)$ as follows (see Figure \ref{fig:ellcurve}). 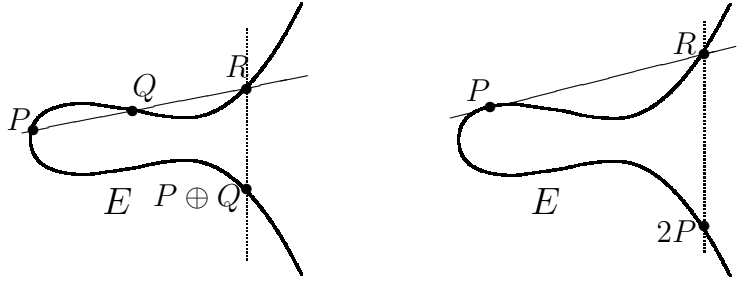
\begin{figure}
\begin{center}
\setlength{\unitlength}{.6pt}
\begin{picture}(200,180)
\put(-150,0){
\begin{picture}(200,180)
\put(-69.5,-59){\usebox{\EllipticCurve}} \thinlines
\put(10,95){\line(5,1){180}} \put(17,97){\circle*6}
\put(0,97){$P$} \put(79.5,109){\circle*6} \put(79.5,118){$Q$}
\put(151.5,122.5){\circle*6} \put(138,130){$R$}
\put(151.5,160){\qbezier[65](0,0)(0,-72.5)(0,-145)}
\put(151.5,59.5){\circle*6} \put(148,45){\makebox(0,0)[br]{$P\oplus
Q$}} \put(70,60){\makebox(0,0)[t]{\large$E$}}
\put(0,0){\makebox(220,-30)[c]{\small Addition of distinct points
$P$ and $Q$}}
\end{picture}
}
\put(120,0){
\begin{picture}(200,180)
\put(-69.5,-59.5){\usebox{\EllipticCurve}}
\put(10,104.5){\line(4,1){180}} \put(35,111){\circle*6}
\put(35,115){\makebox(0,0)[br]{$P$}} \put(170,145){\circle*6}
\put(165,145){\makebox(0,0)[br]{$R$}} \put(170,36){\circle*6}
\put(165,26){\makebox(0,0)[br]{$2P$}}
\put(170,165){\qbezier[65](0,0)(0,-72.5)(0,-145)}
\put(70,60){\makebox(0,0)[t]{\large$E$}}
\put(0,0){\makebox(220,-30)[c]{\small Adding a point $P$ to itself}}
\end{picture}
}
\end{picture}
\end{center}
\caption{\label{fig:ellcurve}The addition law on an elliptic curve.
In the second example the line is tangent to $E$ at $P$.}
\end{figure}
If $P=(x_1, y_1)$ and $Q=(x_2,y_2)$ are in $E(\Q)$, then the line $y=mx+b$ connecting them has rational coordinates.\footnote{We assume the two points are distinct; if they are the same, the argument below must be slightly modified.} Substituting this expression for $y$ into the elliptic curve, we find $(mx+b)^2 = x^3+ax+b$. This is a cubic in $x$ with rational coefficients. By construction two of its roots are $x_1$ and $x_2$, both rational numbers. Thus the third root, say $x_3$, must also be rational. Set $R(P,Q) = (x_3, \sqrt{x_3^3+ax_3+b})$ and $\widetilde{R}(P,Q) = (x_3, -\sqrt{x_3^3+ax_3+b})$. If we define addition by $P\oplus Q = \widetilde{R}(P,Q)$, then this (plus adding a `point at infinity') turns $E(\Q)$ into a finitely generated abelian group. We write $E(\Q)$ as $\Z^r \oplus \mathbb{T}$, where $\mathbb{T}$ is a torsion group\footnote{Mazur \cite{Ma} proved that torsion group is one of the following: $\Z/N \Z$ for $N \in \{1, 2, \dots, 10, 12\}$  or $\Z/2 \times \Z / 2N \Z$ for $N \in \{1, 2, 3, 4\}$.} and $r$ is called the geometric rank of the curve.

Given an elliptic curve $E$ as above, we may associate an $L$-function as follows. Assume $y^2 = x^3 + ax + b$ is a globally minimal Weierstrass equation for $E/\Q$ with discriminant $\Delta = -16(4a^3+27b^2)$ and conductor $N_E$. Set \be a_E(p)\ :=\  p - \#\{(x,y) \in (\Z/p\Z)^2: y^2 \equiv x^3 + ax + b \bmod p\}.\ee Note that the $a_E(p)$'s encode local data, specifically the number of solutions modulo $p$. Hasse proved $|a_E(p)| \le 2 \sqrt{p}$, and we define the $L$-function by \be L(s,E) \ := \ \prod_{p|\Delta} \left(1 - \frac{a_E(p)}{\sqrt{p}} p^{-s}\right)^{-1} \prod_{p\notdiv \Delta} \left(1 - \frac{a_E(p)}{\sqrt{p}} p^{-s} + p^{-2s}\right)^{-1}; \ee we have included the factors of $\sqrt{p}$ so that the completed $L$-function has a functional equation from $s$ to $1-s$ and not $2-s$: \be \Lambda(s,E) \ := \ \left(\frac{\sqrt{N}}{2\pi}\right)^s \Gamma\left(s + \frac12\right) L(s,E) \ = \ \gep_E \Lambda(1-s,E), \ee where $\gep_E \in \{1,-1\}$ is the sign of the functional equation. Following the work of Wiles \cite{Wi}, Taylor-Wiles \cite{TW} and Breuil-Conrad-Diamond-Taylor \cite{BCDT}, we may associate a weight 2 modular form $f$ to any elliptic curve $E$, where the level of $f$ equals the conductor $N_E$ of $E$. We have $\Lambda(s,f) = \Lambda(s,E)$; in particular, the completed $L$-function converges for all $s$. We call the order of vanishing of $\Lambda(s,E)$ at $s=1/2$ the analytic rank of $E$.

The Birch and Swinnerton-Dyer conjecture \cite{BS-D1,BS-D2} states\footnote{There is a more precise form of the conjecture which relates the leading term in the Taylor expansion to the period integral, regulator, Tamagawa numbers and the Tate-Shafarevich group, but this version is not needed for our purposes.} that the order of vanishing of $\Lambda(s,E)$ at the central point $s=1/2$ equals the rank of the Mordell-Weil group $E(\Q)$, or that the analytic rank equals the geometric rank. Sadly, we are far from being able to prove this, though the evidence for the conjecture is compelling, especially in the case of complex multiplication and rank at most 1 \cite{Bro,CW,GKZ,GZ,Kol1,Kol2,Ru}.
In addition there is much suggestive numerical evidence for the conjecture; for example, for elliptic curves with modest geometric rank $r$, numerical approximations of the first $r-1$ Taylor coefficients are consistent with these coefficients vanishing.



\subsection{Explicit Formula}

One powerful tool for investigating the Birch and Swinnerton-Dyer conjecture is the Explicit Formula (see \cite{RS} for a proof for a general $L$-function, or \cite{Mil1} for the calculation for elliptic curves), which connects the zeros of an $L$-functions to the Fourier coefficients.

\begin{thm} Let $\phi$ be an even, twice continuously differentiable test-function whose Fourier transform \be \widehat{\phi}(y) \ := \ \int_{-\infty}^\infty \phi(x) e^{-2\pi i x y} dx \ee has compact support, and denote the non-trivial zeros of $\Lambda(s,E)$ by $\frac12 + i\gamma_E^{(j)}$ (under the Generalized Riemann Hypothesis, each $\gamma_E^{(j)} \in \R$). Then
\begin{eqnarray}\label{eq:thmef}
\sum_{\gamma_E^{(j)}} \phi\left(\gamma_E^{(j)} \frac{\log
N_E}{2\pi}\right) & = & \widehat{\phi}(0) + \phi(0) -  2 \sum_p  \frac{a_E(p)\log p}{p\log N_E}\ \widehat{\phi} \left( \frac{\log p}{\log N_E} \right)  \nonumber\\ & & - 2
\sum_p \frac{a_E^2(p)\log p}{p^2 \log N_E}\  \widehat{\phi} \left(\frac{2 \log p}{\log
N_E} \right)  + O\left(\frac{\log \log
N_E}{\log N_E}\right).\ \ \ \ \ \ \ \ \ \
\end{eqnarray}
\end{thm}

Using the explicit formula, Mestre proved\footnote{Mestre actually proved more, as his results hold for any weight $k$ cuspidal newform, and not just elliptic curves (which correspond to weight 2 cuspidal newforms).}

\begin{thm}[Mestre \cite{Mes}] Assuming the Generalized Riemann Hypothesis: \ben

\item The order of vanishing at the central point is $O(\log N_E / \log\log N_E)$.

\item There is an absolute constant $B$ such that the first zero above the central point occurs before $B/\log \log N_E$.

\een
\end{thm}

From the functional equation, however, we expect the first zero above the central point to be on the order of $1/\log N_E$, and not $1/\log\log N_E$. Thus Mestre's result is significantly larger than what we expect the truth to be; however, it holds for \emph{any} elliptic curve. The situation is very different if instead we consider families of elliptic curves. By averaging the explicit formula over the family and exploiting cancelation in the sums of the Fourier coefficients $a_E(p)$, it is possible to prove (on average) significantly better results.

Numerous studies have been concerned with bounding the average rank in families. We list some of the frequently studied families below (note that, for technical reasons, often one has to do some sieving and remove some curves in order to make certain sums tractable). These results are obtained by averaging the explicit formula over some family $\mathcal{F}_R$, where $R$ is a parameter localizing the conductors, and sending $R\to\infty$.

\bi

\item The family of all elliptic curves: $y^2 = x^3 + Ax + B$, and $\mathcal{F}_R = \{(A,B): |A| \le R^2, |B| \le R^3\}$ (or something along these lines).

\item One parameter families over $\Q(T)$: $y^2 = x^3 + A(T)x+B(T)$, with $A(T)$, $B(T)$ $\in$ $\Z[T]$ and either $\mathcal{F}_R = \{t: R \le t \le 2R\}$ or a sub-family of this where the conductors are given by a polynomial.

\item Quadratic (or higher) twists of a fixed elliptic curve: $d y^2 = x^3 + ax + b$, with $\mathcal{F}_R = \{d: d\le R\ {\rm a\ fundamental\ discriminant}\}$.

\ei

The current record belongs to M. Young \cite{Yo2}, who showed the average rank in the family of all elliptic curves is bounded by $25/14 \approx 1.79$; results for one-parameter families and quadratic twist families are significantly worse. For a sample of the literature, see \cite{BMSW,Bru,BM,CPRW,DFK,Gao,Go,GM,H-B,Kow1,Kow2,Mi,Mil2,RSi,RuSi,Sil3,Yo2,ZK} (especially the surveys \cite{BMSW, Kow1, RuSi}).

\subsection{The one-level density}\label{sec:oneleveldensity}

For a family $\mathcal{F}_R$ of $L$-functions ordered by conductor (with $R\to\infty$), the averaged explicit formula is called the one-level density. Specifically, let $\phi$ be an even Schwartz test-function whose Fourier transform is supported in $(-\sigma, \sigma)$, and denote the zeros of $L(s,f)$ by $1/2+i\gamma_{f,\ell}$ (under GRH each $\gamma_{f,\ell} \in \R$). Let $N_f$ denote the analytic conductor of $L(s,f)$. We define the one-level density by \begin{eqnarray}
D_{\mathcal{F}_R}(\phi)\ :=\ \frac{1}{|\mathcal{F}_R|} \sum_{f\in
\mathcal{F}_R} \sum_{f} \phi\left(\gamma_{f}^{(j)}\frac{\log
N_f}{2\pi}\right).
\end{eqnarray} This statistic has been fruitfully used by many researchers to study the zeros of elliptic curves $L$-functions (as well as other families of $L$-functions) near the central point.

Unlike the $n$-level correlations, which are the same for any cuspidal newform arising from an automorphic representation (see \cite{Hej,Mon,RS}), the one-level density for a family of $L$-functions depends on the symmetry of the family. Katz and Sarnak \cite{KS1, KS2} conjecture that families of $L$-functions correspond to classical compact groups; specifically, the behavior as the conductors tend to infinity of zeros (respectively values) of $L$-functions is well-modeled by the limit as the matrix size tends to infinity of roots (respectively values) of characteristic values of random matrices.\footnote{These conjectures are a natural outgrowth of observed similarities between behavior of $L$-functions and matrix ensembles. While random matrix theory first arose in statistics problems in the early 1900s (see for example \cite{Wis}), it blossomed in the 1950s when it was successfully applied to describe the energy levels of heavy nuclei. Its connections to number theory were first noticed by Montgomery \cite{Mon} and Dyson in the 1970s in studies of the pair correlation of zeros of $\zeta(s)$. See \cite{FM} for a survey on the development of the subject and some of the connections between the two fields.} They conjecture that \be \lim_{R\to\infty} D_{\mathcal{F}_R}(\phi) \ = \ \int \phi(x) W_{G(\mathcal{F})}(x)dx,\ee where $G(\mathcal{F})$ indicates unitary, symplectic or orthogonal (possibly ${\rm SO(even)}$ or ${\rm SO(odd)}$) symmetry; this has been observed in numerous families. Note by Parseval's theorem that \be \int \phi(x)  W_{G(\mathcal{F})}(x)dx \ = \ \int \widehat{\phi}(y) \widehat{W}_{{G(\mathcal{F})}}(y)dy. \ee

Let $I(u)$ be the characteristic function of $[-1,1]$. Katz and Sarnak prove the Fourier transforms of the one-level densities of the classical compact groups are
\begin{eqnarray}
\widehat{W}_{{{\rm SO(even)}} }(u) & = & \delta(u) + \foh I(u) \nonumber\\
\widehat{W}_{{{\rm SO}} }(u) & = & \delta(u) + \foh \nonumber\\
\widehat{W}_{{{\rm SO(odd)}} }(u) & = & \delta(u) - \foh I(u) + 1
\nonumber\\ \widehat{W}_{{{\rm USp}} }(u) & = & \delta(u) - \foh I(u)
\nonumber\\ \widehat{W}_{{{\rm U}} }(u) & = & \delta(u). \end{eqnarray}
For functions whose Fourier Transforms are supported in $[-1,1]$,
the three orthogonal densities are indistinguishable, though they
are distinguishable from $U$ and $Sp$. To detect differences
between the orthogonal groups using the $1$-level density, one
needs to work with functions whose Fourier Transforms are
supported beyond $[-1,1]$.\footnote{One can also distinguish between the various orthogonal groups by looking at the 2-level density, as these three ensembles have distinct behavior for arbitrarily small support; see for instance \cite{Mil2}. If $n \ge 3$, the determinan expansions for the $n$-level density are hard to work with; in fact, in Gao's thesis \cite{Gao} he is able to compute the number theory and random matrix theory results for greater support than he can show agreement. In place of the determinant formulas, one can also use expansions from \cite{HM}; though these hold for smaller support, they are sometimes easier for comparisons.}

For families of elliptic curves with rank, it is useful to consider additional subgroups of the classical compact groups above. We consider the $N\to\infty$ scaling limits of matrices of the form $$\mattwo{I_{r,r}}{}{}{g},$$ where $I_{r,r}$ is the $r\times r$ identity matrix and $g$ is an $N\times N$ orthogonal matrix (drawn from either the full orthogonal family or one of the split families, namely even or odd). These matrices have $r$ forced eigenvalues at 1 (or $r$ eigenangles at 0) for each $g$; thus as we vary $g$ in one of the three families we obtain the same one-level densities as before \emph{except} for an additional factor of $r$. Explicitly,
\begin{eqnarray}
\widehat{W}_{r;{\rm SO(even)}}(u) & = & \delta(u) + \foh I(u) + r \nonumber\\
\widehat{W}_{r;{{\rm SO}} }(u) & = & \delta(u) + \foh + r \nonumber\\
\widehat{W}_{r;{{\rm SO(odd)}} }(u) & = & \delta(u) - \foh I(u) + 1 + r. \end{eqnarray}
For our elliptic curve families, we must evaluate the average over $\mathcal{F}_R$ or $\mathcal{F}_R'$ of \eqref{eq:thmef}. Note that almost all of the conductors will be a bounded power of $R$ for $t \in [R, 2R]$. If we rescale each elliptic curve $E$'s zeros by the correct local factor, namely $(\log N_E)/2\pi$, we have
\begin{eqnarray}\label{eq:thmefaveragedlocal}
D_{\mathcal{F}_R}^{{\rm local}}(\phi) & \ = \ & \frac1{|\mathcal{F}_R|}\sum_{E\in\mathcal{F}_R} \sum_{\gamma_E^{(j)}} \phi\left(\gamma_E^{(j)} \frac{\log
N_E}{2\pi}\right)\nonumber\\ & = & \widehat{\phi}(0) + \phi(0) -  2 \frac1{|\mathcal{F}_R|}\sum_{E\in\mathcal{F}_R}\sum_p  \frac{a_E(p)\log p}{p\log N_E}\ \widehat{\phi} \left( \frac{\log p}{\log N_E} \right)  \nonumber\\ & & - 2 \frac1{|\mathcal{F}_R|}\sum_{E\in\mathcal{F}_R}
\sum_p \frac{a_E^2(p)\log p}{p^2 \log N_E}\  \widehat{\phi} \left(\frac{2 \log p}{\log
N_E} \right)  + O\left(\frac{\log \log
R}{\log R}\right).\ \ \ \ \ \ \ \ \ \
\end{eqnarray} The difficulty with this expression is that, as the conductors are varying, we cannot easily pass the sum over the family through the test-function to the Fourier coefficients $a_E(p)$ and $a_E(p)^2$. By sieving it is possible to surmount these technical details; this is the main result in \cite{Mil2}.

If instead we rescale each elliptic curve $E$'s zeros by the global factor, namely \be\frac{\log N}{2\pi}\ =\ \frac1{|\mathcal{F}_R|} \sum_{t \in \mathcal{F}_R} \frac{\log N_E}{2\pi}, \ee then we find \begin{eqnarray}\label{eq:thmefaveragedglobal}
D_{\mathcal{F}_R}^{{\rm global}}(\phi) & \ = \ & \frac1{|\mathcal{F}_R|}\sum_{E\in\mathcal{F}_R} \sum_{\gamma_E^{(j)}} \phi\left(\gamma_E^{(j)} \frac{\log
N}{2\pi}\right)\nonumber\\ & = & \widehat{\phi}(0) + \phi(0) -  2 \frac1{|\mathcal{F}_R|}\sum_{E\in\mathcal{F}_R}\sum_p  \frac{a_E(p)\log p}{p\log N}\ \widehat{\phi} \left( \frac{\log p}{\log N} \right)  \nonumber\\ & & - 2 \frac1{|\mathcal{F}_R|}\sum_{E\in\mathcal{F}_R}
\sum_p \frac{a_E^2(p)\log p}{p^2 \log N}\  \widehat{\phi} \left(\frac{2 \log p}{\log
N} \right)  + O\left(\frac{\log \log
N}{\log N}\right).\ \ \ \ \ \ \ \ \ \
\end{eqnarray} The analysis is significantly easier here, as now we can pass the summation over the family past the test-function and exploit cancelation in sums of the Fourier coefficients $a_E(p)$ and $a_E(p)^2$.

We quote the best known results for general one-parameter families.

\begin{thm}[Miller \cite{Mil1, Mil2}]\label{thm:millerthesis}\ \\
\noindent Notation:
\bi

\item  Let $\mathcal{E}$ be a one-parameter family of elliptic curves of geometric rank $r$ over $\Q(T)$.

\item Let $\phi$ be a twice continuously differentiable function\footnote{While the theorem was proved under the assumption that $\phi$ is Schwartz, a careful analysis of the argument reveals it suffices that $\phi$ be twice differentiable.} with ${\rm supp}(\widehat{\phi}) \subset (-\sigma, \sigma)$.

\item Consider the sieved family (see Definition \ref{defi:sievedfamily}), and denote the degree of the conductor polynomial by $m$.

\item Let $G$ denote either ${\rm SO}$, ${\rm SO(even)}$ or ${\rm SO(odd)}$.

\ei

\noindent Assume

\bi

\item If $\mathcal{E}$ is not a rational surface (see Remark \ref{rek:commentstate} for a definition) then assume Tate's conjecture.

\item If the discriminant has an irreducible polynomial factor of degree at least 4, assume either the ABC or the Square-free Sieve Conjecture.

\ei

Then \be D_{\mathcal{F}_R}^{{\rm local}}(\phi) \ = \ \int \hphi(y) \widehat{W}_{r;G}(y)dy \ = \ \left(r + \frac12\right) \phi(0) + \frac12 \widehat{\phi}(0) + \left(\frac{\log\log R}{\log R}\right) \ee provided $\sigma < \min\left(1/2, 2/3m\right)$; a similar result holds for $D_{\mathcal{F}_R}^{{\rm global}}(\phi)$ (without the assumptions that $\mathcal{E}$ satisfies Tate's hypothesis and without assuming either the ABC or Square-free Sieve Conjecture). \end{thm}

\begin{rek} We briefly discuss some consequences and generalizations of the above theorem.
\bi
\item Similar statements hold for quadratic twist families and the family of all elliptic curves.

\item The above result provides support that the zeros of one-parameter families of rank $r$ over $\Q(T)$ are modeled by the scaling limits of orthogonal matrices with $r$ independent eigenvalues of 1.

\item As ${\rm supp}(\hphi) \subset (-1,1)$, the three orthogonal groups have indistinguishable one-level densities. We can see which group correctly models our family by studying the 2-level density, which requires us to understand the distribution of signs of the functional equations in our family.

\ei

\end{rek}


\section{Proof of Theorem \ref{thm:mainbound}}\label{sec:lowerboundnumzeros}

\subsection{Preliminaries}

Before proving Theorem \ref{thm:mainbound}, we first prove general results for the upper and lower bounds in a window of variable size for a general family of $L$-functions. Theorem \ref{thm:mainbound} then follows immediately from Theorem \ref{thm:boundsabovebelow}, Theorem \ref{thm:millerthesis} and the constructions of test-functions satisfying the necessary conditions, which are given below.

\begin{thm}\label{thm:boundsabovebelow} Let $\mathcal{F}_R$ denote a family of $L$-functions, and let $Z_{{\rm ave},\mathcal{F},R}^{({\rm global})}(\tau), Z_{{\rm ave},\mathcal{F},R}^{({\rm local})}(\tau)$ be defined as in Definition \ref{defn:Zavelocglob}. Assume for both normalizations of the zeros that there are constants $a$ and $b$ such that \be\label{eq:density} D_{\mathcal{F}_R}(\phi) \ = \ a \phi(0) + b \widehat{\phi}(0) + O\left(\frac{\log\log R}{\log R}\right), \ee whenever $\phi$ or $\psi$ is a twice continuously differentiable function with Fourier transform supported in $(-\sigma, \sigma)$. If $\phi(x) \ge 0$ for $|x| \le \tau$ and $\phi(x) \le 0$ whenever $|x| \ge \tau$, and if $\phi(x)$ is largest when $x=0$, then
\be Z_{{\rm ave},\mathcal{F},R}^{({\rm global})}(\tau),\ Z_{{\rm ave},\mathcal{F},R}^{({\rm local})}(\tau) \ \ \ge \ a + b\frac{\widehat{\phi}(0)}{\phi(0)} + O\left(\frac{\log\log R}{\phi(0)\log R}\right), \ee while if $\psi(x) \geq 0$ for all $x$ and is monotonically decreasing on $\left(0,\tau\right)$, then \bea  Z_{{\rm ave},\mathcal{F},R}^{({\rm global})}(\tau),\ Z_{{\rm ave},\mathcal{F},R}^{({\rm local})}(\tau) \ \le \ a+\frac{a(\psi(0)-\psi(\tau)) + b\hat{\psi}(0)}{\psi(\tau)} +  O\left(\frac{\log\log R}{\psi(0)\log R}\right). \eea
\end{thm}

\begin{proof} We give the proof for the local rescaling; the global case follows analogously.
As $\phi(x)$ is non-positive for $|x| \ge \tau$, the contribution to the one-level density from the scaled zeros as large or larger than $\tau$ in absolute value is non-positive; thus if we remove these contributions then the one-level density gives the lower bound \be
\frac{1}{|\mathcal{F}_R|}\sum_{f\in \mathcal{F}_R} \sum_{|\gamma_{f}^{(j)}| \le \tau} \phi(\widetilde{\gamma}_{f}^{(j)})\ \ge\ a\phi(0) + b\widehat{\phi}(0) + O\left(\frac{\log \log R}{\log R}\right).\ee As $\phi$ is maximized at 0, we increase the left-hand side above by replacing $\phi(\widetilde{\gamma}_{f}^{(j)})$ with $\phi(0)$; doing so and dividing by $\phi(0)$ yields the claimed bound for $Z_{{\rm ave},\mathcal{F},R}^{({\rm local})}(\tau)$. The upper bound is proved analogously.
\end{proof}

\begin{rek}
These results are of course not of interest unless we are able to construct $\phi$ and $\psi$ satisfying the conditions in Theorem \ref{thm:boundsabovebelow}. For one-parameter families of elliptic curves of rank $r$ over $\Q(T)$, we have $a = r + \frac12$ and $b=1$. \end{rek}

\begin{rek} For test functions whose Fourier transform is supported in $(-1, 1)$, all known one-level densities of families of $L$-functions are in the form of Theorem \ref{thm:boundsabovebelow}, and thus our results are immediately applicable. For some families where the support exceeds $(-1, 1)$ (such as families of cuspidal newforms of square-free level split by sign of the functional equation), a little more work is needed as the functional form of the one-level density is different.\footnote{For the family of Dirichlet characters of prime conductor, the 1-level density is known to be $\widehat{\phi}(0)$ for support is known up to $(-2, 2)$, and thus is of the desired form.} For ease of exposition in this paper we confine ourselves to the $(-1, 1)$ case.
\end{rek}



\subsection{Proof of Theorem \ref{thm:mainbound}}

The main step in the proof of Theorem \ref{thm:mainbound} is showing that our result is non-vacuous by constructing $\phi$ and $\psi$ with the claimed properties. Our construction of $\phi$ is almost surely similar to the construction implicit in Mestre's work \cite{Mes}; see also Hughes and Rudnick \cite{HR}.

%

\begin{proof}[Proof of the Lower Bound in Theorem \ref{thm:mainbound}]
We give the lower bound for the number of zeros in $[-\tau, \tau]$ by constructing a good test function $\phi$. As our results depend on the support of $\hphi$ (which is finite), it is convenient to normalize our test function and express everything in terms of $h$, which we take to be an even, twice continuously differentiable function supported on $(-1, 1)$ and monotonically decreasing on $[0,1)$. For fixed $\sigma, \tau > 0$ let $f(y) = h(2y/\sigma)$, $g(y) = (f \ast f)(y)$ (the convolution\footnote{The convolution is defined by $(A \ast B)(x) = \int_{-\infty}^\infty A(t) B(x-t)dt$.} of $f$ with itself), and let $\phi(x)$ equal the Fourier transform of $g(y) + (2\pi\tau)^{-2} g''(y)$. We must show (i) ${\rm supp}(\widehat{\phi}) \subset (-\sigma, \sigma)$ and (ii) $\phi(x)$ is non-negative for $|x| < \tau$ and non-positive for $|x| > \tau$.

The proof of (i) follows from standard properties of convolution. Specifically, as ${\rm supp}(f) \subset (-\sigma/2, \sigma/2)$, we have ${\rm supp}(g) \subset (-\sigma, \sigma)$.\footnote{We may interpret the relation between $f$ and $g$ as follows. Let $X$ be a random variable with density $f$ supported in $(-\sigma/2, \sigma/2)$. Then $g = f\ast f$ is the density of $X+X$, and is supported in $(-\sigma, \sigma)$.} As the support of $g''$ is contained in the support of $g$ and $\hphi(y) = g(y) + (2\pi \tau)^{-2} g''(y)$, the support of $\hphi$ is contained in $(-\sigma, \sigma)$ as claimed.

For (ii), the Fourier transform of $g''(y)$ is $-(2\pi y)^2 \widehat{g}(y)$ (the Fourier transform converts differentiation to multiplication by $2\pi i x$ in our normalization). Further $g=f \ast f$ implies $g''=f*f''$. Combining the above, we find\footnote{As $\phi$ and $\hphi$ are even, the Fourier transform of the Fourier transform is the original function $\phi(x)$; if $\phi$ were not even, we would have to replace $\phi(x)$ with $\phi(-x)$.} the Fourier transform of $\hphi(y) = g(y) + (2\pi \tau)^{-2} g''(y)$ is $\phi(x) = \widehat{g}(x) \cdot \left(1 - (x/\tau)^2\right)$.

To complete the proof, we must show \be \frac{\widehat{\phi}(0)}{\phi(0)} \ = \
\frac{(\int_0^1 h(u)^2 du)+(\frac{1}{\sigma\tau\pi})^2(\int_0^1 h(u)h''(u) du)}{\sigma(\int_0^1 h(u)du)^2}.\ee
By construction we have
\be
\frac{\widehat{\phi}(0)}{\phi(0)}\ = \ \frac{g(0)+(2\pi\tau)^{-2} g''(0)}{\widehat{g}(0)}.
\ee
Since $g$ is even and monotonically decreasing near the origin (as $g$ has a maximum at 0), $g''(0)<0$. Thus larger values of $\tau$ should decrease the ratio above, at the cost of increasing the size of our window.

From our construction, as $h$ and $f$ are even we have
\be
g(0) \ = \ \int_{-\sigma/2}^{\sigma/2} f(t)^2 dt \ = \ 2\int_{0}^{\sigma/2} h\left(\frac{2t}{\sigma}\right)dt \ = \ \sigma\int_0^1 h(u)^2 du
\ee and
\bea
g''(0)&\ = \ &\int_{-\sigma/2}^{\sigma/2} f(t)f''(t) dt\nonumber\\
			&=&2\int_{0}^{\sigma/2} f(t)f''(t) dt\nonumber\\
			&=&\frac{8}{\sigma^2}\int_0^{\sigma/2} h\left(\frac{2t}{\sigma}\right)h''\left(\frac{2t}{\sigma}\right) dt \ \ \ \left({\rm since\ } f(t) = h\left(\frac{2t}{\sigma}\right),\ f''(t) = \frac{4}{\sigma^2} h\left(\frac{2t}{\sigma}\right)\right) \nonumber\\	&=&\frac{4}{\sigma}\int_0^1 h(u)h''(u) du.
\eea
As the Fourier transform of a convolution is the product of the Fourier transforms, a straightforward calculation yields
\bea
\widehat{g}(0)& \ = \ & \widehat{f}(0) \cdot \widehat{f}(0) \ = \ \sigma^2\left(\int_0^1 h(u) du\right)^2. \\ \nonumber
\eea
Collecting the above equalities, after some elementary algebra we can express the ratio ${\hat{\phi}(0)}/{\phi(0)}$ in terms of $h$ as
\be
\frac{\widehat{\phi}(0)}{\phi(0)}\ = \ \frac{(\int_0^1 h(u)^2 du)+(\frac{1}{\sigma\tau\pi})^2(\int_0^1 h(u)h''(u) du)}{\sigma(\int_0^1 h(u)du)^2}.
\ee If we set this ratio equal to zero (i.e., if we choose $\tau$ so that the numerator vanishes) then
we find\footnote{In obvious notation, we have $\int_0^1 h^2 \ge - (\pi \sigma \tau_{\rm critical})^{-2}\int_0^1 h h''$. We see $\int_0^1 h h'' \le 0$, and thus $\tau_{\rm critical} \ge (-\int_0^1 h^2 / \int_0^1 h h'')^{-1/2} (\pi\sigma)^{-1}$. } that on average there are at least $r+\frac12$ normalized zeros in the band $(-\frac1{\pi C(h)\sigma},\frac1{\pi C(h) \sigma})$, where
\be
C(h) \ = \ \left(-\frac{\int_0^1 h(u)^2du}{\int_0^1 h(u)h''(u)du}\right)^{1/2}.
\ee
\end{proof}

\begin{proof}[Proof of the Upper Bound in Theorem \ref{thm:mainbound}] The proof is similar to that of the lower bound; in particular, once we construct a function $\psi$ with the desired properties then the claim follows immediately from straightforward algebra.

We are thus again reduced to constructing a function with the specified properties. For convenience we construct a $\psi$ which is not Schwartz, but which is twice differentiable; a careful analysis of the proof of Theorem \ref{thm:millerthesis} shows that this suffices, and thus such a $\psi$ is sufficient for our purposes.

Consider the  function $\psi(x)= \left( \frac{\sin{x\pi \sigma}}{x\pi\sigma}\right)^{2}$ with a compactly supported Fourier transform given by
\be\twocase{\psi(y)\ =\ }{\frac{1}{\sigma}\left(1-\frac{|y|)}{\sigma}\right)}{if $y \in (-\sigma,\sigma)$}{0}{if $y \not\in (-\sigma,\sigma)$;} \ee see Figure \ref{fig:plotpsi} for a plot. Away from the origin, the derivative is given by \be \psi'(x) \ =\ \frac{2\sin(\sigma \pi x)}{\sigma \pi x^2}\left(\cos(\sigma \pi x) - \frac{\sin(\sigma \pi x)}{\sigma \pi x}\right).\ee
\begin{figure}
\begin{center}
\scalebox{1}{\includegraphics{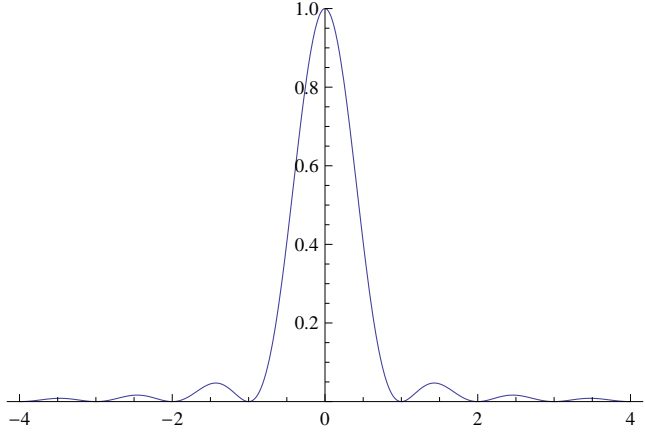}}
\caption{\label{fig:plotpsi} Plot of $\psi(x)$ $=$ $\left(\frac{\sin(x\pi \sigma)}{x\pi\sigma}\right)^2$ for $\sigma = 1$.}
\end{center}
\end{figure} It is easy to see that the global maximum is at $x=0$ and that $\psi(x)$ is decreasing up to $x = 1/\sigma$, proving the claim for any $\tau \le 1/\sigma$ (though the bound worsens as $\tau$ approaches $1/\sigma$ as $\psi(1/\sigma) = 0$).
\end{proof}

\begin{proof}[Proof of the Random Matrix Theory prediction in Theorem \ref{thm:mainbound}]
We assume the conjectures from Random Matrix Theory hold for any even test function, and not just Schwartz test functions. We therefore take $\phi(x)$ to be the characteristic function of the interval $[-\tau, \tau]$, which has Fourier transform equal to $\frac{\sin(2\pi \tau y)}{2\pi \tau y}\cdot 2\tau$. Using such a test function simply counts all normalized zeros in our family that are in $[-\tau, \tau]$ (there is no weighting as $\phi$ is identically 1 in this interval). Thus the predicted average number of such zeros in this interval as $R\to\infty$ is \bea \int_{-\infty}^\infty \widehat{\phi}(y) \widehat{W}_{r;{\rm SO}}(y)dy & \ = \ & \int_{-\infty}^\infty \hphi(y) \left(\delta(y) + \frac12 + r\right)dy \nonumber\\ &=& \left(r + \frac12\right) \phi(0) + \hphi(0) \nonumber\\ &=& r + \frac12 + 2\tau. \eea
\end{proof}

\subsection{Explicit upper and lower bounds}

We conclude by determining the upper and lower bounds from Theorem \ref{thm:mainbound} for the average number of normalized zeros in given intervals as $R\to\infty$.

We first consider the lower bound, which means we must maximize $C(h)$ (as it is in the denominator for $\tau$, the larger $C(h)$ the smaller the window). As the optimal choice of $h$ (in a given class of functions) is only slightly better than similar $h$, we do not spend too much time on determining the truly best $h$. Consider the family of functions given by
\be
h_n(x)\ = \ (1-x^2)(1+a_2x^2+\dots +a_{2i}x^{2i}+\dots+a_{2n}x^{2n}).
\ee
We set $a_0 = 1$ as the maximum is to occur at $x=0$, and since the ratio is invariant under rescaling the $a_i$'s, we might as well take $a_0=1$. Note that each $a_{2i+1} = 0$ as our function is even. We chose $h_n$ of this form as this forces $h_n$ to be even and to vanish at $\pm 1$. We have
\be
C(h_n) \ = \ \left(-\frac{\int_0^1 h_n(u)^2 du}{\int_0^1 h_n(u)h_n''(u) du}\right)^{1/2}.
\ee
The optimum value of the square-root appears to be $2/\pi$. For example, when $n=2$ we must compute \bea & & \max_{a_2,a_4}\left(-\frac{\frac8{15} + \frac{16 a_2}{105} + \frac{8 a_2^2}{315} + \frac{16 a_4}{315} + \frac{16 a_2 a_4}{693} + \frac{8 a_4^2}{1287} }{ -\frac43 - \frac{8 a_2}{15} - \frac{44 a_2^2}{105} - \frac{8 a_4}{35} - \frac{8 a_2 a_4}{15} - \frac{52 a_4^2}{231}}\right)^{1/2}\nonumber\\  & \ = \ & \max_{a_2,a_4} \left(\frac{6006 + 286 a_2^2 + 572 a_4 + 70 a_4^2 + 52 a_2 (33 + 5 a_4)}{39 (385 + 121 a_2^2 + 66 a_4 + 65 a_4^2 +   154 a_2 (1 + a_4))}\right)^{1/2}; \eea this is the quantity inside the square-root, not $C(h)$. Using Mathematica we find the optimal values are $a_2 \approx -.233428$ and $a_4 \approx .0189588$, which leads to $C(h) \approx 0.63662$; as $2/\pi \approx 0.63662$, this suggests the optimal value of $C(h)$ might be $2/\pi$. This yields the window $\left(-\frac{1}{2\sigma}, \frac{1}{2\sigma}\right)$ in which we have on average (as $R\to\infty$) $r+\foh$ zeros.

\begin{rek} As we expect the true answer to be a window of size $0$ (i.e., we expect to be able to take $\sigma = \infty$), it is not worthwhile to find the true optimum above merely to save a bit in a few decimal places. The purpose of this analysis is to show that we do see the correct number of zeros on average in the limit in a window of size proportional to $1/\sigma$; the actual value of the proportionality constant, while interesting, is in some sense immaterial as we believe the density conjecture holds for arbitrary $\sigma$.

We list some approximate values for $C(h)$ for other obvious candidates, which are all less than the 0.63662 (which is approximately $2/\pi$) found above.

\bi

\item $h(x) = (1-x^2)^2$ has $C(h) \approx 0.57735$ (with the quantity inside the square-root looking like $1/3$); if we take just $(1-x^2)$ we get $C(h) = \sqrt{2/5} \approx 0.632456$.

\item $h(x) = \exp(-1/(1-x^2))$ has $C(h) \approx 0.570024$.

\item $h(x) = \exp(-.754212/(1 - x^2))$ has $C(h) \approx 0.575629$ (the value of $.754212$ was obtained by searching for optimal test functions among $\exp(-a/(1-x^2))$).

    \ei

\end{rek}

We now turn to finding explicit upper bounds for the average number of normalized zeros in $[-\tau, \tau]$ as $R\to\infty$. We continue to analyze the candidate function $\psi(x) = \left(\frac{\sin(\pi \sigma x)}{\pi \sigma x}\right)^2$ (see Figure \ref{fig:plotpsi} for a plot). We have freedom in terms of how we rate our approximation; for example, we can decrease the upper bound if we simultaneously decrease the size of the interval.

A natural value to take for the size of our interval is the optimal interval found in the lower bound analysis, namely set $\tau \approx 1/2\sigma$. As \be\twocase{\widehat{\psi}(y) \ = \ }{\frac1{\sigma}\left(1 - \frac{|y|}{\sigma}\right)}{if $|y| \le \sigma$}{0}{otherwise,}\ee we have $\widehat{\psi}(0) = 1/\sigma$, $\psi(0) = 1$ and $\psi(1/2\sigma) = 4/\pi^2 \approx 0.810569$ (or if we believe the approximations, $\psi(\tau) = \frac{\pi^2}{4}\sin^2(2/\pi)$). Thus after some algebra we see that the average number of normalized zeros in the interval $(-\frac{1}{2\sigma}, \frac{1}{2\sigma})$ is at most $\frac{\pi^2}{4}\left(r + \frac12 + \frac1{\sigma}\right)$.

\appendix


\section{Standard Conjectures}\label{sec:standardconj}

At various points in the paper we assume the following conjectures.\\

\noindent \textbf{Generalized Riemann Hypothesis (for Elliptic Curves).}
\emph{Let $\Lambda(s,E)$ be the completed, normalized $L$-function of an elliptic
curve $E$ with function equation $s \to 1-s$. The non-trivial zeros $\rho$ of $\Lambda(s,E)$ have
$\mbox{Re}(\rho) = 1/2$.} \\


\noindent \textbf{Birch and Swinnerton-Dyer Conjecture \cite{BS-D1,BS-D2}.} \emph{Let $E$ be an elliptic curve of geometric rank
$r$ over $\Q$ with Mordell-Weil group  $E(\Q) = \Z^r \oplus \mathbb{T}$. Then
the analytic rank (the order of vanishing of the completed $L$-function at
the critical point) equals the geometric rank.} \\

\noindent \textbf{Tate's Conjecture for Elliptic Surfaces \cite{Ta}.}
\emph{Let $\mathcal{E}/ \Q$ be an elliptic surface and
$L_2(\mathcal{E},s)$ be the $L$-series attached to
$H^2_{\mbox{{\'e}t}}(\mathcal{E}/ \overline{\Q}, \Q_l)$.
$L_2(\mathcal{E},s)$ has a meromorphic continuation to $\C$ and
$-\mbox{ord}_{s=1} L_2(\mathcal{E},s)$ $= \mbox{rank}\
NS(\mathcal{E}/ \Q)$, where $NS(\mathcal{E}/ \Q)$ is the
$\Q$-rational part of the N{\'e}ron-Severi group of $\mathcal{E}$.
Further, $L_2(\mathcal{E},s)$ does not vanish on
the line $\mbox{Re}(s) = 1$.} \\

\begin{rek}\label{rek:commentstate}
Tate's conjecture is known for rational elliptic surfaces. An elliptic surface $\mathcal{E}: y^2 = x^3 + A(T)x + B(T)$ is rational if and only if one of the following is true: $(1)$ $0 < \max\{3\ {\rm deg} A, 2\ {\rm deg} B\} < 12;$ $(2)$ $3\ {\rm deg} A = 2\ {\rm deg} B = 12$ and ${\rm ord}_{T=0}T^{12} \Delta(T^{-1}) = 0$. See \cite{RSi},
pages $46-47$ for more details. \end{rek} \  \\

\noindent \textbf{ABC Conjecture.} \emph{Fix $\epsilon > 0$. For co-prime
positive integers $a$, $b$ and $c$ with $c = a+b$ and $N(a,b,c) =
\prod_{p|abc} p$, $c \ll_\epsilon N(a,b,c)^{1+\epsilon}$.} \\

The full strength of ABC is never needed; rather, we need a
consequence of ABC, the Square-Free Sieve Conjecture (see \cite{Gr}): \\

\noindent \textbf{Square-Free Sieve Conjecture.} \emph{Fix an irreducible
polynomial $f(t)$ of degree at least $4$. As $N \to \infty$, the
number of $t \in [N,2N]$ with $f(t)$ divisible by $p^2$ for some
$p > \log N$ is $o(N)$.} \\

For irreducible polynomials of degree at most $3$, the above is
known, complete with a better error than $o(N)$ (\cite{Ho},
chapter $4$).

We use the Square-Free Sieve to handle the variations in the
conductors. If our evaluation of the logarithm of the conductors is off
by as little as a small constant, the prime sums become
untractable. This is why many works normalize by the average
log-conductor. \\

The following conjecture is used only to interpret some of our results (unless we are calculating the 2-level density to distinguish the three orthogonal candidate groups).\\

\noindent \textbf{Restricted Sign Conjecture (for the Family $\mathcal{F}$).}
\emph{Consider a one-parameter family $\mathcal{F}$ of elliptic
curves. As $N \to \infty$, the signs of the curves $E_t$ are
equidistributed for $t \in [N,2N]$.} \\

The Restricted Sign conjecture can fail (there are
families with constant $j(E_t)$ where all curves have the same
sign, as well as more exotic examples). Helfgott \cite{He} has related the Restricted Sign
conjecture to the Square-Free Sieve conjecture and standard
conjectures on sums of Moebius: \\

\noindent \textbf{Polynomial Moebius.} \emph{Let $f(t)$ be a non-constant
polynomial such that no fixed square divides $f(t)$ for all $t$.
Then $\sum_{t=N}^{2N} \mu(f(t)) = o(N)$.} \\

The Polynomial Moebius conjecture is known for linear $f(t)$.

Helfgott shows the Square-Free Sieve and Polynomial Moebius imply
the Restricted Sign conjecture for many families; this is also discussed in \cite{Mil1}. More precisely,
let $M(t)$ be the product of the irreducible polynomials dividing
$\Delta(t)$ and not $c_4(t)$. \\

\noindent \textbf{Theorem: Equidistribution of Sign in a Family \cite{He}:}
\emph{Let $\mathcal{F}$ be a one-parameter family with $a_i(t) \in
\Z[t]$. If $j(E_t)$ and $M(t)$ are non-constant, then the signs of
$E_t$, $t \in [N,2N]$, are equidistributed as $N \to \infty$.
Further, if we restrict to good $t$, $t \in [N,2N]$ such that
$D(t)$ is good (usually square-free), the signs are still
equidistributed in the limit.}

\ \\


\begin{thebibliography}{BCDDT2}

\bibitem[BMSW]{BMSW}
B. Bektemirov, B. Mazur, W Stein and M. Watkins, \emph{Average ranks of elliptic curves: Tension between data and conjecture}, Bull. Amer. Math. Soc. \textbf{44} (2007), 233--254.

\bibitem[BS-D1]{BS-D1}
\newblock B. Birch and H. Swinnerton-Dyer, \emph{Notes on elliptic
curves. I}, J. reine angew. Math. \textbf{212}, $1963$, $7-25$.

\bibitem[BS-D2]{BS-D2}
\newblock B. Birch and H. Swinnerton-Dyer, \emph{Notes on elliptic
curves. II}, J. reine angew. Math. \textbf{218}, $1965$, $79-108$.

\bibitem[BCDT]{BCDT}
\newblock C. Breuil, B. Conrad, F. Diamond and R. Taylor, \emph{On
the modularity of elliptic curves over \textbf{Q}: wild $3$-adic
exercises}, J. Amer. Math. Soc. \textbf{14}, no. $4$, $2001$,
$843-939$.

\bibitem[Bro]{Bro}
\newblock M. L. Brown, Heegner modules and elliptic curves, Lecture Notes In Mathematics, vol. 1849, Springer-Verlag, 2004.

\bibitem[Bru]{Bru}
A. Brumer, \emph{The average rank of elliptic curves I}, Invent. Math. \textbf{109} (1992) 445-472.

\bibitem[BM]{BM}
A. Brumer and O. McGuinness,  \emph{The behavior of the Mordell-Weil group of	elliptic curves},
Bull. A.M.S.  {\bf 23} (1990) 375-382.

\bibitem[CW]{CW}
J. Coates and A. Wiles, \emph{On the conjecture of Birch and Swinnerton-Dyer},  Invent. Math.  \textbf{39}  (1977), no. 3, 223--251.

\bibitem[CPRW]{CPRW}
J. B. Conrey, A. Pokharel, M. O. Rubinstein and M. Watkins, \emph{Secondary terms in the number of vanishings of quadratic twists of elliptic curve $L$-functions}. In Ranks of elliptic curves and
random matrix theory, pages 215--232, London Math. Soc. Lecture Note Ser. \textbf{341}, Cambridge
Univ. Press, Cambridge, 2007.

\bibitem[DFK]{DFK}
C. David, J. Fearnley and H. Kisilevsky, \emph{On the vanishing of twisted $L$-functions of elliptic curves}, Experiment. Math. \textbf{13} (2004), no. 2, 185--198.

\bibitem[DHKMS]{DHKMS}
E. Due\~nez, D. K. Huynh, J. Keating, S. J. Miller and N. Snaith, \emph{Models for zeros at the central point in families of elliptic curves}, in preparation.

\bibitem[FM]{FM}
F. W. K. Firk and S. J. Miller, \emph{Nuclei, Primes and the Random Matrix Connection}, Symmetry \textbf{1} (2009), 64--105; doi:10.3390/sym1010064.

\bibitem[Gao]{Gao}
\newblock P. Gao, \emph{$N$-level density of the low-lying zeros of
quadratic Dirichlet $L$-functions}, Ph.~D thesis, University of
Michigan, 2005.

\bibitem[Go]{Go}
D. Goldfeld,  \emph{Conjectures on elliptic curves over quadratic fields}, in
Number Theory, Carbondale, Lecture Notes in Mathematics  {\bf 751}, 108-118. Springer-Verlag,  1979.

\bibitem[GM]{GM}
F. Gouv\^ea and B. Mazur, 	\emph{The square-free sieve and the rank of elliptic curves},
J. AMS {\bf 4} (1991) 1-23.

\bibitem[Gr]{Gr}
\newblock Granville, \emph{ABC Allows Us to Count
Squarefrees}, International Mathematics Research Notices
\textbf{19}, $1998$, $991-1009$.

\bibitem[GKZ]{GKZ}
\newblock B. H. Gross, W. Kohnen and D. B. Zagier, \emph{Heegner points and derivatives of L-series. II}, Mathematische Annalen \textbf{278} (1987), no. 1–4, 497--562.

\bibitem[GZ]{GZ}
\newblock B. H. Gross and D. B. Zagier, \emph{Heegner points and derivatives of L-series}, Inventiones Mathematicae \textbf{84} (1986), no. 2, 225--320.


\bibitem[H-B]{H-B}
\newblock D. R. Heath-Brown, \emph{The average rank of elliptic curves IV},
Duke Math. J. \textbf{122} (2004), no. 3, 591--623.

\bibitem[Hej]{Hej}
\newblock D. Hejhal, \emph{On the triple correlation of zeros of
the zeta function}, Internat. Math. Res. Notices 1994, no. 7,
294-302.


\bibitem[He]{He}
H. A. Helfgott, \emph{On the behaviour of root numbers in families of elliptic curves}, preprint,
2004. \texttt{http://arxiv.org/abs/math/0408141}

\bibitem[Ho]{Ho}
    \newblock C. Hooley, \emph{Applications of Sieve Methods to the
Theory of Numbers}, Cambridge University Press, Cambridge, $1976$.


\bibitem[HM]{HM}
\newblock C. Hughes and S. J. Miller, \emph{Low-lying zeros of $L$-functions
with orthogonal symmtry}, Duke Math. J., \textbf{136}
(2007), no. 1, 115--172.

\bibitem[HR]{HR}
\newblock C. Hughes and Z. Rudnick, \emph{Linear Statistics of
Low-Lying Zeros of $L$-functions},  Quart. J. Math. Oxford
\textbf{54} (2003), 309--333.

\bibitem[Huy]{Huy}
\newblock D. K. Huynh, \emph{Elliptic curve $L$-functions
of finite conductor and random matrix theory}, PHD Thesis, University of Bristol, 2009.

\bibitem[KS1]{KS1}
\newblock N. Katz and P. Sarnak, \emph{Random Matrices, Frobenius
Eigenvalues and Monodromy}, AMS Colloquium Publications
\textbf{45}, AMS, Providence, $1999$.

\bibitem[KS2]{KS2}
\newblock N. Katz and P. Sarnak, \emph{Zeros of zeta functions and symmetries},
Bull. AMS \textbf{36}, $1999$, $1-26$.

\bibitem[Kn]{Kn}
\newblock A. Knapp, \emph{Elliptic Curves}, Princeton University Press,
Princeton, $1992$.

\bibitem[Kob]{Kob}
\newblock N. Koblitz, \emph{Introduction to Elliptic Curves and Modular Forms}, Springer-Verlag, 1993.

\bibitem[Kol1]{Kol1}
\newblock V. A. Kolyvagin, \emph{The Mordell-Weil and Shafarevich-Tate groups for Weil elliptic curves}, Izv. Akad. Nauk SSSR Ser. Mat.  \textbf{52}  (1988),  no. 6, 1154--1180, 1327;  translation in  Math. USSR-Izv.  \textbf{33}  (1989),  no. 3, 473--499

\bibitem[Kol2]{Kol2}
\newblock V. A. Kolyvagin, \emph{Finiteness of $E(Q)$ and ${\rm Shah}(E,Q)$ for a subclass of Weil curves}, Izv. Akad. Nauk SSSR Ser. Mat. \textbf{52} (1988), no. 3, 522--540, 670--671; translation in Math. USSR-Izv. \textbf{32} (1989), no. 3, 523--541.

\bibitem[Kow1]{Kow1}
E. Kowalski, \emph{Elliptic curves, rank in families and random matrices}. In Ranks of Elliptic Curves and Random Matrix Theory, London Mathematical Society Lecture Note Series (No. 341), edited by J. B. Conrey, D. W. Farmer, F. Mezzadri and N. C. Snaith, 2007.

\bibitem[Kow2]{Kow2}
E. Kowalski, \emph{On the rank of quadratic twists of elliptic curves over function fields}, International J. Number Theory, 2006.

\bibitem[Ma]{Ma}
B. Mazur, \emph{Rational isogenies of prime degree}, Inventiones Math. \textbf{44} (1978), no. 2,  129--162.

\bibitem[Mes]{Mes}
\newblock J. Mestre, \emph{Formules explicites et minorations de conducteurs
 de vari\'{e}t\'{e}s alg\'{e}briques}, Compositio Mathematica \textbf{58} (1986),
 209--232.

\bibitem[Mi]{Mi}
\newblock P. Michel, \emph{Rang moyen de familles de courbes elliptiques
et lois de Sato-Tate}, Monat. Math. \textbf{120} (1995), $127-136$.

\bibitem[Mil1]{Mil1}
S. J. Miller, \emph{$1$- and $2$-Level Densities for Families of Elliptic
Curves: Evidence for the Underlying Group Symmetries}, P.H.D.
Thesis, Princeton University, $2002$.\hfill \\ \texttt{http://www.williams.edu/go/math/sjmiller/public\underline{\ }html/math/\\ thesis/SJMthesis\underline{\ }Rev2005.pdf.}

\bibitem[Mil2]{Mil2}
S. J. Miller, \emph{$1$- and $2$-level densities for families of
elliptic curves: Evidence for the underlying group symmetries},
Compositio Mathematica \textbf{104} (2004), no. 4, 952--992.

\bibitem[Mil3]{Mil3}
S. J. Miller, \emph{Investigations of zeros near the central point of
elliptic curve $L$-functions} (with an appendix by E. Due\~nez), Experimental Mathematics \textbf{15}
(2006), no. 3, 257--279.


\bibitem[Mon]{Mon}
\newblock H. Montgomery, \emph{The pair correlation of zeros of the zeta
function}, Analytic Number Theory, Proc. Sympos. Pure Math.
\textbf{24}, Amer. Math. Soc., Providence, $1973$, $181-193$.

\bibitem[RSi]{RSi}
\newblock M. Rosen and J. Silverman, \emph{On the rank of an elliptic
surface}, Invent. Math. \textbf{133} (1998), $43-67$.

\bibitem[Ru]{Ru}
\newblock K. Rubin, \emph{The one-variable main conjecture for elliptic curves with complex multiplication,
$L$-functions and arithmetic} (Durham, 1989), in London Math. Soc. Lecture Note Series \textbf{153},
Cambridge Univ. Press, Cambridge, 1991, pages 353--371.

\bibitem[RuSi]{RuSi}
K. Rubin and A. Silverberg, \emph{Ranks of elliptic curves}, Bull.  Amer.  Math. Soc.  {\bf 39} (2002) 455-474.

\bibitem[RS]{RS}
\newblock Z. Rudnick and P. Sarnak, \emph{Zeros of principal $L$-functions
 and random matrix theory}, Duke Journal of Math. \textbf{81},
 $1996$, $269-322$.

\bibitem[Sil1]{Sil1}
J. Silverman, \emph{The Arithmetic of Elliptic Curves}, Graduate
Texts in Mathematics, Vol. 106, Springer-Verlag, New York, 1986.

\bibitem[Sil2]{Sil2}
\newblock J. Silverman, \emph{Advanced Topics in the Arithmetic of
Elliptic Curves}, Graduate Texts in Mathematics \textbf{151},
Springer-Verlag, Berlin - New York, 1994.

\bibitem[Sil3]{Sil3}
\newblock J. Silverman, \emph{The average rank of an algebraic
family of elliptic curves}, J. reine angew. Math. \textbf{504}
(1998), 227--236.

\bibitem[ST]{ST}
J. Silverman and J. Tate, \emph{Rational Points on Elliptic
Curves}, Springer-Verlag, New York, 1992.

\bibitem[Ta]{Ta}
\newblock J. Tate, \emph{Algebraic cycles and the pole of zeta
functions}, Arithmetical Algebraic Geometry, Harper and Row, New
York, $1965$, $93-110$.

\bibitem[TW]{TW}
\newblock R. Taylor and A. Wiles, \emph{Ring-theoretic properties of
certain Hecke algebras}, Ann. Math. \textbf{141}, $1995$,
$553-572$.

\bibitem[Wi]{Wi}
\newblock A. Wiles, \emph{Modular elliptic curves and Fermat's last
theorem}, Ann. Math \textbf{141}, $1995$, $443-551$.

\bibitem[Wis]{Wis}
J. Wishart, \emph{The generalized product moment distribution in
samples from a normal multivariate population}, Biometrika
\textbf{20 A} (1928), 32--52.

\bibitem[Yo1]{Yo1}
M. P. Young, \emph{Basics of elliptic curves}, talk at the American Institute of Mathematics,
June 1, 2006. \texttt{http://www.aimath.org/conferences/ntrmt/talks/\\ BasicsofEllipticCurves.pdf}.

\bibitem[Yo2]{Yo2}
M. P. Young, \emph{Low-lying zeros of families of elliptic curves},
J. Amer. Math. Soc. \textbf{19} (2006), no. 1, 205--250.

\bibitem[ZK]{ZK}
D. Zagier and G. Kramarz, \emph{Numerical investigations related to the $L$-series of certain elliptic curves},
J. Indian Math. Soc. {\bf 52} (1987) 51-69.

\end{thebibliography}
\end{document}